\documentclass[12pt,a4paper]{amsart}
\usepackage{graphicx}
\usepackage{color}
\usepackage{subfigure}
\usepackage{amsmath, amssymb, amsthm, amsfonts}

\allowdisplaybreaks
\usepackage{amsmath,amssymb,amsthm,amsfonts,mathrsfs,enumerate,sansmath,mathtools,amscd
}
\usepackage[colorlinks=true]{hyperref}
\hypersetup{urlcolor=blue, citecolor=blue, draft=false}

\newtheorem{theorem}{Theorem}[section]
\newtheorem{proposition}[theorem]{Proposition}

\newtheorem{corollary}[theorem]{Corollary}
\theoremstyle{definition}
\newtheorem{definition}[theorem]{Definition}
\newtheorem{example}[theorem]{Example}
\newtheorem{remark}[theorem]{Remark}

\begin{document}
\title[A generalized Hurwitz metric] {A generalized Hurwitz metric}

\author[Arstu and S. K. Sahoo]{Arstu$^\dagger$ and Swadesh Kumar Sahoo$^\dagger$\,$^*$}

\address{$^\dagger$Discipline of Mathematics, Indian Institute of Technology Indore, 
Simrol, Khandwa Road, Indore 453 552, India}
\email{arstumothsra@gmail.com}
\email{swadesh.sahoo@iiti.ac.in}


\thanks{*~The corresponding author}
\begin{abstract}
In 2016, the Hurwitz metric was introduced by D. Minda in arbitrary proper subdomains of the complex plane and he proved that this metric coincides with the 
Poincar\'e's hyperbolic metric when the domains are simply connected. In this paper, we provide an alternate definition of the Hurwitz metric through which
we could define a generalized Hurwitz metric in arbitrary subdomains of the complex plane. This paper mainly highlights various 
important properties of the Hurwitz metric and the generalized metric including the situations where they coincide with each other.     
\end{abstract}

\maketitle
\section{Introduction and Preliminaries}\label{sec1}
\setcounter{equation}{0}

In $19^{\rm th}$ century, the notion of hyperbolic metric was first introduced. As pointed out, for instance in \cite[p.~132]{ke07} and \cite{Min16}, the hyperbolic density on a hyperbolic domain $\Omega$ can be understood through the extremal problem of maximizing $|f'(0)|$ over all holomorphic functions $f$ that map the unit disk into $\Omega$. In 1981, Hahn \cite{Hahn81} introduced a pseudo-differential metric for complex manifolds by means of an extremal problem. 
Two years later, Minda \cite{Min83} reconsidered the Hahn metric in Riemann surfaces. Recently, Minda considered an extremal problem of Hurwitz \cite{Hur1904} and introduced a new conformal metric, namely, the Hurwitz metric \cite{Min16}
in any proper subdomain of the complex plane $\mathbb{C}$. Our objective in this paper is to
investigate further properties of the Hurwitz metric
and their applications. 

In $2007,$ Keen and Lakic \cite{Kee07} defined some new densities in arbitrary plane domains that generalize the hyperbolic density. They are namely the generalized Kobayashi density (see \cite[Definition~2]{Kee07}) and the generalized Carath\'eodory density (see \cite[Definition~9.2, p. 166]{ke07}). The Kobayashi density is defined by pushing forward the hyperbolic density from the unit disk to a plane domain by a holomorphic function, whereas, the Carath\'eodory density is defined by pulling back the hyperbolic density from a plane domain to the unit disk by a holomorphic function. This paper deals with a generalized Hurwitz metric
in the sense of Kobayashi. We are considering a generalized Hurwitz 
metric in the sense of Carath\'eodory in our next paper.  
Since holomorphic functions are infinitesimal contractions in the hyperbolic metric, it is easy to see that the generalized Kobayashi density exceeds over hyperbolic density on hyperbolic domains. Furthermore, the hyperbolic and the generalized Kobayashi densities coincide whenever there is a regular holomorphic covering map \cite[p~125]{ke07} from the source domain to the range domain. Similar to the case of hyperbolic distance the generalized Kobayashi distance, in
association with the generalized Kobayashi density, between two points can be defined by taking infimum of the generalized Kobayashi length of all rectifiable paths joining the points. In fact, with this definition, it becomes a complete metric space. 

As an analogue of the generalized Kobayashi density, we shall generalize the Hurwitz
metric by pushing forward the Hurwitz density from a proper subdomain of the complex plane to an arbitrary domain by holomorphic functions with some specific properties. We call this new density {\em the generalized Hurwitz density}. Note that, on hyperbolic domains the Hurwitz density exceeds the hyperbolic density. Furthermore, in this work we prove that the generalized Hurwitz density is always greater than the Hurwitz density.

Throughout this article, our notations are relatively standard and we are working mainly on the complex plane $\mathbb{C}.$ First we denote the open unit disk by $\mathbb{D}:=\{w\in\mathbb{C}:|w|<1\}.$ 
The classical Hyperbolic density \cite[p.~33]{ke07} in $\mathbb{D}$ is defined as 
$$
\lambda_{\mathbb{D}}(w)=\frac{2}{1-|w|^2}
$$ 
for $w\in\mathbb{D}$.
Note that we consider the hyperbolic metric with constant curvature $-1$.
The hyperbolic distance between two points $w_1,w_2$ in $\mathbb{D}$ is
$$
\lambda_{\mathbb{D}}(w_1,w_2)=\inf\int_{\gamma}\lambda_{\mathbb{D}}(w)\,|dw|,
$$
where the infimum is taken over all paths $\gamma$ joining $w_1$ and $w_2$ 
in $\mathbb{D}$. Since the hyperbolic metric is conformal invariant, 
by the Riemann Mapping Theorem one can easily define it on the proper simply 
connected domains of $\mathbb{C}$. 
However, this metric is also defined in more general domains so-called hyperbolic domains.
A plane domain $\Omega$ is called hyperbolic if $\mathbb{C}\setminus\Omega$ contains at least two points.
On a hyperbolic domain $\Omega$, 
the hyperbolic density $\lambda_{\Omega}$ \cite[p.~124]{ke07}  is 
$$
\lambda_{\Omega}(w)=\frac{\lambda_{\mathbb{D}}(t)}{|\pi'(t)|},
$$ 
where $\pi:\mathbb{D}\to\Omega$ is a universal covering map with $\pi(t)=w$. 
Analogue to the case of the unit disk, the hyperbolic distance between two points 
$w_1$ and $w_2$ in $\Omega$ is defined by
$$
\lambda_{\Omega}(w_1,w_2)=\inf\int_{\gamma}\lambda_{\Omega}(w)\,|dw|,
$$
where the infimum is taken over all paths $\gamma$ joining $w_1$ and $w_2$ in $\Omega$.

We now define some notations that are used in the definition of the Hurwitz density defined in \cite{Min16}.

Throughout this paper we denote by $\mathcal{H}(Y,\Omega)$ for the set of all holomorphic functions from a domain $Y$ to another domain $\Omega.$
For a fixed point $s\in\mathbb{D}$, let $\mathcal{H}_s(w,\Omega)$ be the family of all holomorphic functions $h$ from $\mathbb{D}$ into $\Omega$ with $h(s)=w,~h(t)\neq w$ for all $t\in\mathbb{D}\setminus\{s\}$ and $h'(s)>0.$ 
For a point $w\in\Omega$, we write
$\mathcal{H}(w,\Omega):=\mathcal{H}_0(w,\Omega)$.
The \textit{Hurwitz density} \cite{Min16} in $\Omega$ is defined as
$$
\eta_{\Omega}(w)=\frac{2}{G'(0)}=\frac{2}{r_\Omega(w)},
$$ 
where $G'(0)=\max\{h'(0):~h\in\mathcal{H}(w,\Omega)\}=:r_\Omega(w).$ 
If $\Omega\subsetneq\mathbb{C}$ is a domain and $b\in\Omega$ is any point, then the covering map $G:\mathbb{D}\setminus\{0\}\rightarrow\Omega\setminus\{b\}$ extends to a holomorphic function $G_b:\mathbb{D}\rightarrow\Omega$ with $G_b(0)=b,~G_b'(0)>0$. The extended holomorphic function $G_b$ is called the {\em Hurwitz covering} \cite{Min16} of $(\mathbb{D},0)$ onto $(\Omega,b).$ 

The distance decreasing property, which is stated below, of the Hurwitz density for the holomorphic function plays a crucial role to prove our results in this article.

\medskip
\noindent{\bf Theorem A}\label{thmB} (Distance decreasing property of the Hurwitz density). \cite{Min16}
{\em Suppose that $\Omega$ and $\triangle$ are proper subdomains of $\mathbb{C},~a\in\Omega$ and $b\in\triangle.$ If $h$ is a holomorphic function of $\Omega$ into $\triangle$ with $h(a)=b$ and $h(w)\neq b$ for $w\in\Omega\setminus\{a\},$ then
$$
\eta_{\triangle}(b)|h'(a)|\leq\eta_{\Omega}(a).
$$ 
Moreover, equality holds if and only if $h$ is a covering of $\Omega\setminus\{a\}$ onto $\triangle\setminus\{b\}$ that extends to a holomorphic map of $\Omega$ onto $\triangle$ with $h(a)=b$ and $h'(a)\neq 0.$}
 
The Structure of this document is organized as follows. In Section $\ref{sec3}$, certain basic properties of the Hurwitz density are studied leading to the concept of Hurwitz distance which produces the completeness property of the metric space.
Furthermore, Section~$\ref{sec4}$ deals with some basic properties of the generalized Hurwitz density. Finally, in Section $\ref{sec5}$, we relate the Hurwitz and the the generalized Hurwitz densities over some specific plane domains having certain geometric properties.

\section{The Hurwitz Metric}\label{sec3}
\setcounter{equation}{0}

Firstly, we present here a characterization of the Hurwitz density which gives us ideas to introduce the notion of generalized Hurwitz density in the next section. Let $F$ be the extremal function for the extremal problem
 $
 \max\{h'(s): h\in\mathcal{H}_s(w,\Omega)\}.
 $ 
 Then, we have 
\begin{align*}
 F'(s)&=\max\{h'(s): h\in\mathcal{H}_s(w,\Omega)\}\\
     &=\max\{h'(s)=(f\circ T)'(s): f\in\mathcal{H}(w,\Omega)~\mbox{and $T$, the M\"obius transformation}\\ 
     & \hspace*{6.5cm}\mbox{of $\mathbb{D}$ onto itself with $T(s)=0$ and $T'(s)>0$}\}.
\end{align*}
Since $T(z)=(z-s)/(1-\overline{s}z)$, it follows that
$$
F'(s)=\max\left\{\cfrac{f'(0)}{1-|s|^2}:~f\in\mathcal{H}(w,\Omega)\right\}.
$$
 Since the hyperbolic density on $\mathbb{D}$ is given by $\lambda_{\mathbb{D}}(s)=2/(1-|s|^2),$ by the notations defined in the previous section, we have
 $$
 F'(s)=\cfrac{\lambda_{\mathbb{D}}(s)}{\eta_{\Omega}(w)}.
 $$
 By using this argument, we provide here an alternate definition of the Hurwitz density as follows:

\begin{definition}\label{def1}
The {\em Hurwitz density} on a proper subdomain $\Omega$ of $\mathbb{C}$ is defined as
\begin{equation}{\label{eq1}}
\eta_{\Omega}(w)=\cfrac{\eta_{\mathbb{D}}(s)}{g'(s)},
\end{equation}
where $g=h\circ T$ such that $T$ is the M\"obius transformation from $\mathbb{D}$ onto $\mathbb{D}$ with $T(s)=0,~T'(s)> 0$ and $h$ is the Hurwitz covering map from $\mathbb{D}$ onto $\Omega$ with $h(0)=w.$
\end{definition}

To define the Hurwitz distance between any two points in $\Omega$ we integrate the density $\eta_{\Omega}$ and obtain the following definition:

\begin{definition}\label{def2}[Hurwitz distance]
For $w_1,w_2$ in $\Omega,$ we define 
$$
\eta_{\Omega}(w_1,w_2)=\inf\int_{\gamma}\eta_{\Omega}(w)|dw|,
$$ 
where infimum is taken over all rectifiable paths $\gamma$ in $\Omega$ joining $w_1$ and $w_2.$
\end{definition}

Note that we are using the same notation for the Hurwitz density as well as the Hurwitz distance between any two points where the distinction can be observed by seeing the number of parameters. However, now onward, for simplicity, we sometimes use the notation $\eta_{\Omega}$ for $\eta_{\Omega}(w_1,w_2).$ To justify our above definition we indeed prove that $\eta_{\Omega}$ defines a metric when the domain $\Omega$ is assumed to be hyperbolic.

\begin{theorem}\label{thm1}
If $\Omega$ is a hyperbolic domain, then $(\Omega,\eta_{\Omega})$ is a complete metric space.
\end{theorem}

\begin{proof}
 By the definition of $\eta_{\Omega},$ symmetry and triangle inequality follow directly. Therefore to prove that $(\Omega,\eta_{\Omega})$ is a metric space, we need to prove strictly positivity of the Hurwitz distance between any two distinct points.
 Let $w_1,w_2$ be any two distinct points in $\Omega$. Since $\eta_{\Omega}(w_1,w_2)$ is the infimum of the Hurwitz length of all rectifiable curves joining $w_1$ and $w_2$ in $\Omega$, for any $\epsilon>0$ there exists a rectifiable path $\gamma$ such that
 \begin{equation*}
 \eta_{\Omega}(w_1,w_2)\geq\int_{\gamma}\eta_{\Omega}(w)|dw|-\epsilon.
 \end{equation*}
 Note that in a hyperbolic domain $\Omega$, the inequality $\eta_{\Omega}\geq\lambda_{\Omega}$ is
 well-known; see \cite{Min16}. Then we have
 \begin{equation*}
  \eta_{\Omega}(w_1,w_2)\geq\int_{\gamma}\lambda_{\Omega}(w)|dw|-\epsilon
  \geq\lambda_{\Omega}(w_1,w_2)-\epsilon.
 \end{equation*}
 Letting $\epsilon\to 0$, we obtain
 \begin{equation}\label{eqn1}
 \eta_{\Omega}(w_1,w_2)\geq\lambda_{\Omega}(w_1,w_2)>0.
 \end{equation}

To prove the completeness, we use the following fact (see \cite[Theorem~2.5.28, p.~52]{Bur01}):
{\em a locally compact length $($metric$)$ space $X$ is complete if and only if every closed disc in $X$ is compact $($see also \cite[p.~28]{Bur01}$)$}. Because $\lambda_{\Omega}$ is complete, each closed hyperbolic disk $\overline{D}_{\lambda_{\Omega}}(a,r)=\{w\in\Omega:\,\lambda_{\Omega}(a,w)\leq r\}$ is compact. Because $\lambda_{\Omega}\leq\eta_{\Omega},~\overline{D}_{\eta_{\Omega}}(a,r)\subset\overline{D}_{\lambda_{\Omega}}(a,r)$. A closed subset of a compact set is compact, so $\overline{D}_{\eta_{\Omega}}(a,r)$ is compact.
\end{proof}

The following remark assures that
there exists a non-hyperbolic domain for which Theorem~\ref{thm1} still satisfies.

\begin{remark}
Let $\Omega=\mathbb{C}\setminus\{0\}$. Then, the Hurwitz density has the elementary formula 
$\eta_{\Omega}(w)=1/8|w|.$ This is nothing but
a scalar multiplication of the classical 
quasihyperbolic metric of $\Omega$.
The completeness property now follows from the fact that the quasihyperbolic metric space is complete.
\end{remark}

We know that the holomorphic functions are global as well as infinitesimal contraction functions with respect to the hyperbolic metric. In analogy to this we now prove that the one-to-one holomorphic functions are global contraction functions for the Hurwitz metric as well. 

\begin{proposition}\label{prop1}
Let $\Omega$ and $Y$ be proper subdomains of $\mathbb{C}$ and $h$ be an injective holomorphic function from $\Omega$ to $Y.$ Then we have the inequality
$$
\eta_Y(h(w_1,w_2))\leq\eta_{\Omega}(w_1,w_2)
$$
for all $w_1,w_2$ in $\Omega.$ Equality holds in the above inequality if $h$ is an conformal homeomorphism.
\end{proposition}

\begin{proof}
By definition of $\eta_\Omega(w_1,w_2)$, for any $\epsilon>0$ there exists a path $\gamma$ joining $w_1$ and $w_2$
in $\Omega$ such that 
$$
\int_{\gamma}\eta_{\Omega}(w)|dw|\le \eta_{\Omega}(w_1,w_2)+\epsilon.
$$
By Definition $\ref{def2},$ it follows clearly that
\begin{equation}\label{eq2}
\eta_{Y}(h(w_1,w_2))\leq\int_{h(\gamma)}\eta_Y(z)|dz|=
\int_{\gamma}\eta_Y(h(w))|h'(w)||dw|.
\end{equation}
Since $h$ is one-to-one holomorphic function, by Theorem~A, we have
\begin{equation}\label{eq3}
\eta_Y(h(w))|h'(w)|\leq\eta_{\Omega}(w)
\end{equation}
for every $w$ in $\Omega.$ 
Combining $\eqref{eq2}$ and $\eqref{eq3},$ we obtain
$$
\eta_{Y}(h(w_1,w_2))\leq\int_{\gamma}\eta_{\Omega}(w)|dw|\le \eta_{\Omega}(w_1,w_2)+\epsilon.
$$
Letting $\epsilon\to 0$, we conclude what we wanted to prove.
\end{proof}

\section{The generalized Hurwitz Metric}\label{sec4}
\setcounter{equation}{0}

In Section \ref{sec3} we discussed the alternate definition of the Hurwitz density. By adopting the idea of generalized Kobayashi density we are going to define and study the generalized Hurwitz density in this section. 
The distance decreasing property of the Hurwitz density implies that 
for any holomorphic function $h$ from $\mathbb{D}$ to $\Omega$ with 
$h(s)=w,~ h(t)\neq w$ for all $t$ in $\mathbb{D}\setminus\{s\}$ and 
$h'(s)\neq 0,$ we have the inequalities
\[ \eta_{\Omega}(h(s))|h'(s)| \leq\eta_{\mathbb{D}}(s),\]
and
\[ \eta_{\Omega}(h(s)) \leq\cfrac{\eta_{\mathbb{D}}(s)}{|h'(s)|}. \]
Since the formula $\eqref{eq1}$ provides an existence of a holomorphic function $h$ for which the equality holds, we have
$$
\eta_{\Omega}(w)=\inf\cfrac{\eta_{\mathbb{D}}(s)}{|h'(s)|},
$$
where the infimum is taken over all holomorphic functions $h$ from $\mathbb{D}$ to $\Omega$ with $h(s)=w,~h(t)\neq w$ for all $t\in\mathbb{D}\setminus\{s\},~h'(s)\neq 0.$ This leads to the notion of introducing generalized Hurwitz density for an arbitrary domain $\Omega$.

\begin{definition}\label{def3}
For any domain $\Omega\subset\mathbb{C},$ the {\em generalized Hurwitz density} is defined as 
\begin{equation*}
\eta_{\Omega}^{\mathbb{D}}(w)=\inf\cfrac{\eta_{\mathbb{D}}(s)}{|h'(s)|},
\end{equation*} 
where the infimum is taken over all $h\in\mathcal{H}(\mathbb{D},\Omega)$ with 
$h(s)=w,~h(t)\neq w$ for all $t\in\mathbb{D}\setminus\{s\},~h'(s)\neq 0,$ and all $s$ in $\mathbb{D}.$
\end{definition}

\begin{remark}
If $\Omega\subsetneq\mathbb{C}$ and $a\in\Omega,$ then there exists a Hurwitz covering map $g$ from $\mathbb{D}$ to $\Omega$ which realizes the infimum; thus $\eta_{\Omega}(a)=\eta_{\Omega}^{\mathbb{D}}(a)$ for all $a\in\Omega.$   
\end{remark}

Note that, in Definition $\ref{def3}$ it is not required to choose the domain $\Omega$ to be a proper subdomain of the complex plane $\mathbb{C}$. In the following theorem, we calculate $\eta_{\Omega}^{\mathbb{D}},$ when $\Omega=\mathbb{C}.$    

\begin{theorem}\label{thm2}
Suppose that $\Omega$ is the whole complex plane, then the generalized Hurwitz density $\eta_{\Omega}^{\mathbb{D}}(w)$ is identically equal to zero for all elements $w$ in $\Omega.$
\end{theorem}

\begin{proof}
Let $w\in\Omega$ be an arbitrary element and $n$ be any positive integer. Setting $h_n(s)=(s-t)n+w.$ Clearly, $h_n$ is a sequence of holomorphic functions from $\mathbb{D}$ into $\mathbb{C}$ with $h_n(t)=w$ and $h_n(s)\neq w$ for all $s\in\mathbb{D}\setminus\{t\}.$ By Definition $\ref{def3},$ we have 
\begin{equation*}
\eta_{\Omega}^\mathbb{D}(w)\leq\cfrac{\eta_\mathbb{D}(t)}{|h'_n(t)|}.
\end{equation*}
Since the Hurwitz and the hyperbolic densities coincide on simply connected domains, we have
\begin{equation*}
\eta_{\Omega}^{\mathbb{D}}(w)\leq\cfrac{\lambda_{\mathbb{D}}(t)}{n}.
\end{equation*}
Letting $n$ goes to infinity, we obtain that $\eta_{\Omega}^{\mathbb{D}}(w)=0.$
\end{proof}	

The definition of the generalized Hurwitz density can further be generalized by changing the fixed domain $\mathbb{D}$ to an arbitrary proper subdomain $Y$ of $\mathbb{C},$ that is, by pushing forward the Hurwitz density on $Y$ to $\Omega$ by a holomorphic function having some special property. Here, we call $Y$ as the {\em basepoint domain}. This idea leads to the following definition.

\begin{definition}\label{def4}
Let $\Omega\subset \mathbb{C}$ be arbitrary. For
all $s\in Y$,
the {\em generalized Hurwitz density} $\eta_{\Omega}^Y$ for the basepoint domain $Y$ is defined as
$$
\eta_{\Omega}^Y(w)=\inf\cfrac{\eta_Y(s)}{|h'(s)|},
$$
where $\eta_Y$ is the Hurwitz density on $Y$ and the infimum is taken over all holomorphic functions $h$ from $Y$ to $\Omega$ with $h(s)=w,~h(t)\neq w$ for all $t\in Y\setminus\{s\},~h'(s)\neq0$. 
\end{definition}

In view of the nature of Definition~\ref{def4}, it is here appropriate to remark that $\eta_{\Omega}^Y$ can be $+\infty$ at some points, or even at every point.

We will now prove some expected elementary properties of $\eta_{\Omega}^Y.$ We start by comparing the Hurwitz and the generalized Hurwitz densities on proper subdomains of the complex plane.

\begin{proposition}\label{prop2}
Let $Y\subset \mathbb{C}$ be a domain and $\Omega$ be a proper subdomain of $\mathbb{C}$. Then for every point $w$ in $\Omega,$ we have
\begin{equation*}
\eta_{\Omega}^Y(w)\geq\eta_{\Omega}(w).
\end{equation*}
\end{proposition}
\begin{proof}
Let $a\in Y$ and $h$ be any holomorphic function from $Y$ to $\Omega$ with
$h(a)=b,~h(s)\neq b$ for all $s\in Y\setminus\{a\}$ and $h'(a)\neq 0.$ Then by distance decreasing property of the Hurwitz density, we have
\[ \eta_{\Omega}(h(a))|h'(a)| \leq\eta_Y(a), \]
and
\[ \eta_{\Omega}(b) \leq\cfrac{\eta_Y(a)}{|h'(a)|}.\]
Taking the infimum on both sides over $h\in\mathcal{H}(Y,\Omega)$ with $h(a)=b,~h(s)\neq b$ for all $s\in Y\setminus\{a\},~h'(a)\neq 0$, we have
$$
\eta_{\Omega}^Y(b)\geq\eta_{\Omega}(b).
$$
Since $b\in\Omega$ is arbitrary, the above inequality holds true for every $b\in\Omega.$
\end{proof}

One naturally asks the comparison between the classical generalized Kobayashi density and the generalized Hurwitz density. Recall the definition of the generalized Kobayashi density.

\begin{definition}
Let $\Omega$ be a domain in the complex plane. For every $w\in \Omega$, the {\em generalized Kobayashi density} is
given by 
$$
\kappa_\Omega^Y(w)=\inf \frac{\lambda_Y(t)}{|f'(t)|},
$$
where $\lambda_Y$ is the hyperbolic density on a hyperbolic domain $Y\subset\mathbb{C}$ and the infimum is taken over all $f\in\mathcal{H}(Y,\Omega)$ and all points $t\in Y$ such that $f(t)=w$.
\end{definition}

We immediately have
\begin{corollary}
If $Y$ and $\Omega$ are hyperbolic domains, then $\eta_{\Omega}^Y\geq\kappa_{\Omega}^Y$.
\end{corollary}

\begin{corollary}
For proper subdomains $\Omega$ and $Y$ of $\mathbb{C}$, we have $\eta_{\Omega}^Y(w)\geq 0$ for all $w$ in $\Omega.$
\end{corollary}

There are certain situations where the classical generalized Kobayashi density  agrees with the hyperbolic density, see for instance \cite{Kee07}.
This motivates us to investigate the situations under which the generalized Hurwitz density coincides with the Hurwitz density. The following proposition justifies one such case and a few more situations will be covered in the next section.  

\begin{proposition}\label{prop3}
Let $\Omega$ and $Y$ be proper subdomains of $\mathbb{C}.$ If for every $b\in\Omega,$ there exists a holomorphic covering map $h_b$ from $Y\setminus\{a\}$ onto $\Omega\setminus\{b\}$ that extends to a holomorphic map of $Y$ onto $\Omega$ with $h_b(a)=b$ and $h'_b(a)\neq 0,$ then
$$
\eta_{\Omega}^Y(w)=\eta_{\Omega}(w)
$$
for all $w$ in $\Omega.$
In particular, we also have
$$
\eta_{\Omega}^{\Omega}(w)=\eta_{\Omega}(w)
$$
for every $w$ in $\Omega.$
\end{proposition}

\begin{proof}
Since $h_b$ is a holomorphic map from $Y$ to $\Omega$ with $h_b(a)=b,~h_b(w)\neq b$ for all $w$ in $Y\setminus\{a\}$ and $h'_b(a)\neq 0$, by the definition of generalized Hurwitz density, we have
\begin{equation}\label{eq4}
\eta_{\Omega}^Y(b)\leq\cfrac{\eta_Y(a)}{|h'_b(a)|.}
\end{equation}
In addition, by Theorem~A, we obtain
\begin{equation}\label{eq5}
\eta_{\Omega}(b)|h'_b(a)|=\eta_Y(a).
\end{equation}
Combining $\eqref{eq4}$ and $\eqref{eq5},$ we obtain
$$
\eta_{\Omega}^Y(b)\leq\eta_{\Omega}(b).
$$
Since $b$ is an arbitrary point, it follows that $\eta_{\Omega}^Y(w)\leq\eta_{\Omega}(w)$ for all $w\in\Omega$. On the other hand, by Proposition $\ref{prop2}$ it follows that $\eta_{\Omega}^Y(w)\geq\eta_{\Omega}(w).$ Hence the proof is complete. 
\end{proof}

Proposition~\ref{prop3} is stronger, because 
for non-simply connected domains $Y$ and $\Omega$ the proposition
certainly holds (see for instance Example~\ref{exp3.18}). To demonstrate this, we use 
the distance decreasing property of the generalized Hurwitz density, which is proved below (see Theorem~\ref{thm3}).
However, for simply connected domains we have the following special situation. 

\begin{corollary}
If $Y\subsetneq \mathbb{C}$ is a simply connected domain and $\Omega\subsetneq\mathbb{C}$ is any domain, then 
$$
\eta_{\Omega}^\mathbb{D}\equiv\eta_{\Omega}^Y\equiv\eta_{\Omega}.
$$
\end{corollary}

\begin{proof}
Since $Y\subsetneq\mathbb{C}$ is a simply connected domain, by Riemann Mapping Theorem there exists a conformal homeomorphism $T$ from $Y$ onto $\mathbb{D}.$ Furthermore, $\Omega\subsetneq\mathbb{C}$ implies that for every point $w\in\Omega$ there is a Hurwitz covering map $g_w$ from $\mathbb{D}$ onto $\Omega$ with $g_w(0)=w.$ Hence, by using the composed map $g\circ T$ from $Y$ onto $\Omega$ in Proposition $\ref{prop3},$ we conclude our result.
\end{proof}

It is well-known that both the Hurwitz and the hyperbolic metrics as well the generalized Kobayashi metric $\kappa$ have distance decreasing properties. The following result provides a similar property for the generalized Hurwitz metric.

\begin{theorem}{\rm(Distance decreasing property of the generalized Hurwitz density)}\label{thm3}
Let $\Omega$ and $\triangle$ be any subdomains of $\mathbb{C}$ and $Y\subsetneq\mathbb{C}$ be a domain. If $h$ is a holomorphic function from $\Omega$ to $\triangle$ with $h(a)=b,~h'(a)\neq 0 $ and $h(w)\neq b$ for all $w\in\Omega\setminus\{a\},$ then
$$
\eta_{\triangle}^Y(h(a))|h'(a)|\leq\eta_{\Omega}^Y(a).
$$
\end{theorem}

\begin{proof}
By the definition of generalized Hurwitz density, for every $\epsilon>0$ there exists a point $c\in Y$ and a holomorphic map $g$ from $Y$ to $\Omega$ with $g(c)=a,~g(s)\neq a$ for all $s\in Y\setminus\{c\},~g'(c)\neq 0$
and
\begin{equation}\label{eq6}
\eta_{\Omega}^Y(a)\geq\cfrac{\eta_Y(c)}{|g'(c)|}-\epsilon.
\end{equation}
Note that, $h\circ g$ maps $Y$ to $\triangle$ such that $(h\circ g)(c)=b,~ (h\circ g)(s)\neq b$ for all $s\in Y\setminus\{c\}$ and
$(h\circ g)'(c)=h'(g(c))g'(c)=h'(a)g'(c)\neq 0.$ Therefore, using 
$h\circ g$ in the definition of $\eta_{\triangle}^Y(h(a)),$ we have
$$
\eta_{\triangle}^Y(h\circ g)(c)\leq\cfrac{\eta_Y(c)}
{|(h\circ g)'(c)|}.
$$
By $\eqref{eq6}$ and using the chain rule, it follows that
$$
\eta_{\triangle}^Y(h\circ g)(c)|h'(a)|\leq\cfrac{\eta_Y(c)}{|g'(c)|}\leq\eta_{\Omega}^Y(a)+\epsilon.
$$
Letting $\epsilon$ goes to zero, we obtain
$$
\eta_{\triangle}^Y(h(a))|h'(a)|\leq\eta_{\Omega}^Y(a).
$$
This completes the proof.
\end{proof}

We now provide an example which demonstrate 
Proposition~\ref{prop3} in non-simply connected
domains. 
 
\begin{example}\label{exp3.18}
Let $Y=\mathbb{D}^*:=\mathbb{D}\setminus\{0\}$, the punctured unit disk and 
$\Omega=\mathbb{C}^*:=\mathbb{C}\setminus\{0\}$, the punctured plane. We shall prove that for all $w\in\mathbb{C}^*$
$$
\eta_{\mathbb{C}^*}^{\mathbb{D}^*}(w)=\eta_{\mathbb{C}^*}(w).
$$ 

As stated in \cite[p. 322]{Neh52} (see also \cite[(4.1)]{Min16}), the Hurwitz covering map from $\mathbb{D}$ onto $\mathbb{C}\setminus\{1\}$ is obtained by the infinite product representation
\begin{equation}\label{InfProd}
g(w)=16w\prod_{n=1}^{\infty}\Big(\frac{1+w^{2n}}{1+w^{2n-1}}\Big)^8,\quad |w|<1.
\end{equation} 
For $s\in\mathbb{D}$, it is well known that 
the map $T(z)={(z-s)}/{(1-\overline{s}z)}$ defines a M\"obius transformation of $\mathbb{D}$ onto itself. Clearly, $T(s)=0,T'(s)={(1+|s|^2)}/{(1-|s|^2)}>0.$ Since $g$ is a holomorphic covering map and $T$ is a one-one holomorphic map on $\mathbb{D}$, 
the composition $g\circ T$ is also a holomorphic covering map satisfying 
$(g\circ T)(s)=0$ and $(g\circ T)(t)\neq 0 $ for all $t\in\mathbb{D}\setminus\{s\}.$
Furthermore,  $(g\circ T)'(s)>0$ which follows from the chain rule and the fact that 
$g'(0)=16>0$ and $T'(s)>0.$
By the distance decreasing property we have 
\begin{equation}\label{eqExp3.12}
\eta_{\mathbb{C}\setminus\{1\}}((g\circ T)(s))(g\circ T)'(s)=\eta_{\mathbb{D}}(s).
\end{equation}
Restricting the function $g\circ T$ onto $\mathbb{D}^*$ and plugging it in the 
definition of $\eta_{\mathbb{C}\setminus\{1\}}^{\mathbb{D}^*}(0)$ we obtain
$$
\eta_{\mathbb{C}\setminus\{1\}}^{\mathbb{D}^*}(0)\leq\cfrac{\eta_{\mathbb{D}^*}(s)}{(g\circ T)'(s)}.
$$
Now, we choose a sequence $s_n\in\mathbb{D}^*$ such that $|s_n|\to 1.$ 
By using the same argument as above, we can find M\"obius transformations 
$T_n$ from $\mathbb{D}$ onto itself with $T(s_n)=0$ and $T'(s_n)> 0.$ 
Therefore, it follows from \eqref{eqExp3.12} that 
$$
\eta_{\mathbb{C}\setminus\{1\}}^{\mathbb{D}^*}(0)
\leq\cfrac{\eta_{\mathbb{D}^*}(s_n)}{(g\circ T_n)'(s_n)}
=\cfrac{\eta_{\mathbb{D}^*}(s_n)}{\eta_{\mathbb{D}}(s_n)}\eta_
{\mathbb{C}\setminus\{1\}}(0),
$$
since $(g\circ T)'(s_n)=0$.
We notice from \cite[Section 2]{Min16} that the Hahn density of the punctured unit disk obtained by (see \cite[(2)]{Min83}) 
$$S_{\mathbb{D}^*}(s_n)
=\cfrac{1+|s_n|}{4|s_n|(1-|s_n|)}
$$ 
exceeds 
the Hurwitz density. Thus, we obtain
$$
\eta_{\mathbb{C}\setminus\{1\}}^{\mathbb{D}^*}(0)
\le \cfrac{\eta_{\mathbb{D}^*}(s_n)}{\eta_{\mathbb{D}}(s_n)}
\eta_{\mathbb{C}\setminus\{1\}}(0)
\leq\cfrac{S_{\mathbb{D}^*}(s_n)}{\eta_{\mathbb{D}}(s_n)}
\eta_{\mathbb{C}\setminus\{1\}}(0)
=\cfrac{1+|s_n|}{4|s_n|(1-|s_n|)}(1-|s_n|^2)
\eta_{\mathbb{C}\setminus\{1\}}(0).
$$
Now, letting $|s_n|\to 1$ we have $\eta_{\mathbb{C}\setminus\{1\}}^{\mathbb{D}^*}(0)
\leq\eta_{\mathbb{C}\setminus\{1\}}(0).$ The reverse inequality is followed by 
Proposition \ref{prop2}. Now, by using the holomorphic functions 
$f(w)=1-w$ and $h(w)=bw$ (for some complex constant $b$) in the distance decreasing property for the 
generalized Hurwitz density, it follows that, both the metrics coincide 
on $\mathbb{C}^*$. That is, 
$\eta_{\mathbb{C}^*}^{\mathbb{D}^*}(w)=\eta_{\mathbb{C}^*}(w)$ for all $w\in \mathbb{C}^*.$
\hfill{$\Box$}
\end{example}

Next we define the generalized Hurwitz distance between two points in a domain.

\begin{definition}
Let $\Omega\subset\mathbb{C}$ and $Y\subsetneq\mathbb{C}$ be domains. For $w_1,w_2\in\Omega$, define
$$
\eta_{\Omega}^Y(z_1,z_2)=\inf\int_{\gamma}\eta_{\Omega}^Y(w)|dw|,
$$ 
where the infimum is taken over all rectifiable paths $\gamma$ in $\Omega$ joining 
$z_1$ to $z_2.$
\end{definition}

Proof of the following theorem is similar to that of Theorem~\ref{thm1}.
\begin{theorem}
Let $Y\subsetneq\mathbb{C}$ be a domain.
If $\Omega$ is a hyperbolic domain, then $(\Omega,\eta_{\Omega}^Y)$ is a complete metric space.
\end{theorem}

We do have also the distance decreasing property in the global sense whose proof follows the steps of the proof of Proposition~\ref{prop1}.
\begin{theorem}
Let $Y$ be a proper subdomain of $\mathbb{C}$ and $\Omega,~ \triangle$ be any subdomain of $\mathbb{C}.$ If $h$ is a one-to-one holomorphic map from $\Omega$ to $\triangle$, then
$$
\eta_{\triangle}^Y(h(w_1,w_2))\leq\eta_{\Omega}^Y(w_1,w_2),
$$
for all $w_1,w_2$ in $\Omega.$
\end{theorem}

Note that, till now we have derived all the results of the generalized Hurwitz density $\eta_{\Omega}^Y$ for a base domain $Y.$ In the next theorem we will see the comparison between generalized Hurwitz densities when the range domain is fixed while the source domain is varying. 

\begin{theorem}\label{thm4}
Let $Y_1,~Y_2$ be proper subdomains of $\mathbb{C}$ and $\Omega$ be any subdomain of $\mathbb{C}.$ If for every point $b\in Y_2,$ there exists a point $a\in Y_1$ and a holomorphic covering map $h_b$ from $Y_1\setminus\{a\}$ onto $Y_2\setminus\{b\}$ which extends to a holomorphic map from $Y_1$ onto $Y_2$ with $h_b(a)=b,~ h_b(w)\neq b$ for any $w\in Y_1\setminus\{a\}$ and $h'_b(a)\neq 0,$ then
$$
\eta_{\Omega}^{Y_1}(\zeta)\leq\eta_{\Omega}^{Y_2}(\zeta),
$$
for all $\zeta$ in $\Omega.$
\end{theorem}

\begin{proof}
 Let $\zeta$ be any arbitrary point in $\Omega$ and $\epsilon$ be a positive real number.  By definition of $\eta_{\Omega}^{Y_2},$ there exists a point $b$ in $Y_2$ and a holomorphic function $g$ from $Y_2$
to $\Omega$ with $g(b)=\zeta,~g(s)\neq \zeta$ for any $s\in Y_2\setminus\{b\},~g'(b)\neq 0$ such that
\begin{equation}\label{eq7}
\eta_{\Omega}^{Y_2}(\zeta)\geq\cfrac{\eta_{Y_2}(b)}{|g'(b)|}-\epsilon.
\end{equation}
Since for every point $b\in Y_2,$ there exists a point $a$ in $Y_1$ and a holomorphic covering $h_b$ from $Y_1$ onto $Y_2$ with $h_b(a)=b,~h_b(w)\neq b$ for any $w\in Y_1\setminus\{a\},~g'(a)\neq 0$, by Theorem~A we have
\begin{equation}\label{eq8}
\eta_{Y_2}(b)|h'_b(a)|=\eta_{Y_1}(a).
\end{equation}
Note that, the composition $g\circ h_b$ is a holomorphic function from $Y_1$ to $\Omega$ with $(g\circ h_b)(a)=\zeta,~ (g\circ h_b)(w)\neq \zeta$ for any $w$ in 
$Y_1$ and $(g\circ h_b)'(a)=g'(b)h'_b(a)\neq 0.$ Therefore, by the definition of $\eta_{\Omega}^{Y_1},$ we obtain
\begin{equation}\label{eq9}
\eta_{\Omega}^{Y_1}(a)\leq\cfrac{\eta_{Y_1}(a)}{|(g\circ h_b)'(a)|}
=\cfrac{\eta_{Y_1}(a)}{|g'(b)h'_b(a)|}.
\end{equation}
By $\eqref{eq7},~\eqref{eq8},~\eqref{eq9},$ it follows that
$$
\eta_{\Omega}^{Y_2}(\zeta)\geq\cfrac{\eta_{Y_2}(b)}{|g'(b)|}-\epsilon
=\cfrac{\eta_{Y_1}(a)}{|g'(b)||h'_b(a)|}-\epsilon
\geq\eta_{\Omega}^{Y_1}(\zeta)-\epsilon.
$$   
Letting $\epsilon$ goes to zero, we have $\eta_{\Omega}^{Y_2}(\zeta)\geq\eta_{\Omega}^{Y_1}(\zeta),$ which completes the proof our result.
\end{proof}

We look forward for the existence of 
non-simply connected domains $Y_1$ and $Y_2$
validating the statement of Theorem~\ref{thm4},
however, they remain open due to their non-trivial
nature.

\begin{corollary}
If $Y_1\subsetneq\mathbb{C}$ is a simply connected domain, then for all proper subdomains $Y_2$ and $\Omega$ of $\mathbb{C}$, we have
$$
\eta_{\Omega}(w)=\eta_{\Omega}^{\mathbb{D}}(w)=\eta_{\Omega}^{Y_1}(w)\leq
\eta_{\Omega}^{Y_2}(w),
$$
for all $w$ in $\Omega.$
\end{corollary}

\begin{proof}
Since $Y_1\subsetneq\mathbb{C}$ is a simply connected domain, by Riemann Mapping Theorem, there exists a conformal homeomorphism $f$ from $Y_1$ onto $\mathbb{D}.$ Furthermore, there exists a Hurwitz covering map $T_b$  from $\mathbb{D}$ onto $Y_2$ for every $b$ in $Y_2.$ Thus, the composed map $T_b\circ f$ is a holomorphic covering from $Y_1\setminus\{f^{-1}(0)\}$ onto $Y_2\setminus\{b\},$ which extends to a holomorphic function from $Y_1$ to $Y_2$ with $(T_b\circ f)(f^{-1}(0))=0,~ (T_b\circ f)'(f^{-1}(0))\neq 0$ and $(T_b\circ f)(s)\neq b$ for all $s$ in 
$Y_2\setminus\{f^{-1}(0)\}.$ Taking $h_b=T_b\circ f$ in Theorem $\ref{thm4},$ we obtain the desired result.
\end{proof}
Two subdomains $Y_1$ and $Y_2$ of $\mathbb{C}$ are conformally equivalent 
if there exists a holomorphic bijection $f$ from $Y_1$ to $Y_2$.

\begin{corollary}\label{cor4.17}
Let $\Omega\subset\mathbb{C}$ be any arbitrary domain. If $Y_1\subsetneq\mathbb{C}$ and $Y_2\subsetneq\mathbb{C}$ are conformally equivalent domains, then 
$$
\eta_{\Omega}^{Y_1}(w)=\eta_{\Omega}^{Y_2}(w)
$$
for all $w$ in $\Omega.$ 
\end{corollary}


\section{Lipschitz Domain}\label{sec5}
\setcounter{equation}{0}

In this section, one of our main objectives is to study the situations, in terms of the Hurwitz non-Lipschitz domains, when the Hurwitz density coincides with the generalized Hurwitz density. The following notations are useful in the definition of Hurwitz Lipschitz domains. Let $Y$ be a hyperbolic domain and $\Omega$ be a subdomain of $Y$. If $i$ is the inclusion map from 
$\Omega$ to $Y$, the global contraction constant $gl_{\eta}(\Omega,Y)$ is defined by
$$
gl_{\eta}(\Omega,Y):=\sup_{w_1,w_2\in\Omega,~ w_1\neq w_2}
\cfrac{\eta_Y(w_1,w_2)}{\eta_{\Omega}(w_1,w_2)}.
$$
If $Y$ is any proper subdomain of $\mathbb{C}$, then the infinitesimal contraction constant is defined as
$$
l_{\eta}(\Omega,Y):=\sup_{w\in\Omega}\cfrac{\eta_Y(w)}{\eta_{\Omega}(w)}.
$$
Since the inclusion map is an $injective$ holomorphic function from $\Omega$ to $Y,$ by the distance decreasing property of Hurwitz density, we have $\eta_Y(w)\leq\eta_{\Omega}(w)$ for every $w$ in $\Omega.$ Furthermore, by Proposition $\ref{prop1}$ it follows that $\eta_Y(w_1,w_2)\leq\eta_{\Omega}(w_1,w_2)$ for all $w_1$ and $w_2$ in $\Omega.$ Thus, both infinitesimal and global contraction constants are less than or equal to 1. 

\begin{theorem}\label{thm6}
Let $Y\subsetneq\mathbb{C}$ be a domain. If $\Omega$ is a subdomain of $Y,$ then $gl_{\eta}(\Omega,Y)\leq l_{\eta}(\Omega,Y)\leq 1.$ Furthermore, the inclusion map $i$ from $\Omega$ to $Y$ is a strict infinitesimal contraction map, whenever $\Omega$ is a proper subdomain of $Y.$  
\end{theorem} 

\begin{proof}
Let $w_1,w_2$ be any two points in $\Omega$ and $\gamma\subset\Omega$ be any path joining $w_1$ and $w_2$ such that 
$$
\eta_{\Omega}(w_1,w_2)=\int_{\gamma}\eta_{\Omega}(w)|dw|.
$$
By the definition of $\eta_Y(w_1,w_2),$ it follows that
$$
\eta_{Y}(w_1,w_2)\leq\int_{\gamma}\cfrac{\eta_Y(w)}{\eta_{\Omega}(w)}\,
\eta_{\Omega}(w)|dw|\leq l_{\eta}(\Omega,Y)\,\eta_{\Omega}(w_1,w_2).
$$
Thus, $gl_{\eta}(\Omega,Y)\leq l_{\eta}(\Omega,Y).$

The proof of the second part of our theorem follows from \cite[Theorem~6.1]{Min16}.
\end{proof}

\begin{definition}
Let $\Omega$ be a subdomain of a domain $Y$ in $\mathbb{C}$. Then $\Omega$ is called a {\em Hurwitz Lipschitz subdomain} of $Y,$ if the inclusion map from $\Omega$ to $Y$ is a strict infinitesimal contraction. That is, the infinitesimal contraction constant $l_{\Omega}$ is strictly less than $1.$   
\end{definition}

By Proposition $\ref{prop1},$ for any proper subdomain $\Omega$ of $\mathbb{C},$ we have $\eta_{\Omega}^Y\geq\eta_{\Omega}.$ However, in the following theorem, we find a condition on $Y$ so that for every proper subdomain
$\Omega$ of $\mathbb{C}$, the Hurwitz and the generalized Hurwitz densities coincide.
We adopt the proof technique from \cite[Theorem~2.1]{Tav09} where the author compares the hyperbolic density with the Kobayashi density. 

\begin{theorem}\label{thm7}
If $\Omega$ is any proper subdomain of $\mathbb{C}$ and $Y$ is a Hurwitz non-Lipschitz subdomain of $\mathbb{D},$ then we have
$$
\eta_{\Omega}^{Y}(w)=\eta_{\Omega}(w),
$$  
for every $w$ in $\Omega.$
\end{theorem}

\begin{proof}
To show that $\eta_{\Omega}^Y=\eta_{\Omega},$ we only need to show that $\eta_{\Omega}^Y\leq\eta_{\Omega},$ as the relation $\eta_{\Omega}^Y\geq\eta_{\Omega}$ always holds. Since $\Omega\subsetneq\mathbb{C},$ for every point $w\in\Omega$ there exists a Hurwitz covering map $g$ from $\mathbb{D}$ onto $\Omega$ with $g(0)=w,~g(s)\neq w$ for all $s\in\mathbb{D}\setminus\{0\}$ and $g'(0)\neq 0.$ We now pre-compose $g$ with a M\"obius transformation $T$ of $\mathbb{D}$ that maps $s$ in $Y$ to the origin. Therefore $g\circ T$ is a holomorphic covering of $\Omega$ from $\mathbb{D}\setminus\{s\}$ onto $\Omega\setminus\{w\},$ which extends to a holomorphic function from 
$\mathbb{D}$ to $\Omega$ with $(g\circ T)(s)=w,~(g\circ T)(t)\neq w$ for all $t$ in $\mathbb{D}\setminus\{s\}$ and $(g\circ T)'(s)=g'(0)T'(s)\neq 0.$ Thus, by Theorem~A, we have
\begin{equation}\label{eq11}
\eta_{\Omega}(w)|(g\circ T)'(s)|=\eta_{\mathbb{D}}(s).
\end{equation}
Now, let $h$ be the restriction of $g\circ T$ on $Y.$ By the definition of $\eta_{\Omega}^Y,$ we obtain
$$
\eta_{\Omega}^Y(w)\leq\cfrac{\eta_{Y}(s)}{|h'(s)|}=
\cfrac{\eta_{Y}(s)}{|(g\circ T)'(s)|}.
$$     
On the other hand, by the help of $\eqref{eq11},$ we have 
$$
\eta_{\Omega}^Y(w)\leq\cfrac{\eta_{Y}(s)}{\eta_{\mathbb{D}}(s)}
\eta_{\Omega}(w).
$$
Since $Y$ is a non-Lipschitz Hurwitz subdomain of $\mathbb{D},$ by choosing $s$ in $Y$ appropriately, $\eta_Y(s)/\eta_{\mathbb{D}}(s)$ can be made as close to $1$ as we wish. Thus, we can say that
$$
\eta_{\Omega}^Y(w)\leq\eta_{\Omega}(w).
$$ 
Since $w\in\Omega$ is an arbitrary element, therefore we have $\eta_{\Omega}^Y(w)=\eta_{\Omega}(w) $ for all $w$ in $\Omega.$
\end{proof} 

In order to generalize Theorem $\ref{thm7} $ for a broader class of domains in $\mathbb{C}$, we now discuss the notion of quasi-bounded domains as follows. 

\begin{definition}
A domain $Y$ in $\mathbb{C}$ is said to be {\em quasi-bounded} if the smallest simply connected plane domain containing $Y$ is a proper subset of $\mathbb{C}$. We denote the smallest simply connected domain by $\hat{Y}.$
\end{definition}

\begin{example}
All bounded domains in $\mathbb{C}$ are quasi-bounded.
\end{example}  

The following theorem is an analog of \cite[Theorem~2.2]{Tav09} from the notion of hyperbolic density to the notion of Hurwitz density.

\begin{theorem}\label{thm8}
If $\Omega$ is any proper subdomain of the complex plane and $Y$ is quasi-bounded, non-Lipschitz Hurwitz subdomain of $\hat{Y},$ then
$$
\eta_{\Omega}^Y(w)=\eta_{\Omega}(w),
$$ 
for all $w$ in $\Omega.$
\end{theorem}

\begin{proof}
Since, $\hat{Y}$ is simply connected, by Riemann Mapping Theorem there exists a conformal homeomorphism $h$ from $\hat{Y}$ onto $\mathbb{D.}$ Note that the conformal mappings are isometries for the Hurwitz metric (see \cite[Corollary~6.2]{Min16}), and thus we have
\begin{equation}\label{eq12}
\eta_{\mathbb{D}}(h(s))|h'(s)|=\eta_{\hat{Y}}(s)
\end{equation}
for all $s$ in $\hat{Y}.$ Furthermore, the restriction of $h$ to $Y,$ resulting a conformal homeomorphism from $Y$ onto $h(Y).$ Therefore, we have
\begin{equation}\label{eq13}
\eta_{h(Y)}(s)|h'(s)|=\eta_{Y}(s).
\end{equation}
 By using $\eqref{eq12},~\eqref{eq13}$ and the definition of $l_{\eta}(Y,\hat{Y}),$ we obtain
\begin{equation*}
l_{\eta}(Y,\hat{Y})=\sup_{z\in Y}\cfrac{\eta_Y(z)}{\eta_{\hat{Y}}(z)}=\sup_{z\in h(Y)}
\cfrac{\eta_{h(Y)}(z)}{\eta_{\mathbb{D}}(z)}=l_{\eta}(h(Y),\mathbb{D}).
\end{equation*}
Thus, $Y$ is a non-Lipschitz Hurwitz subdomain of $\hat{Y}$ if and only if $h(Y)$ is non-Lipschitz Hurwitz subdomain of $\mathbb{D}.$ By Theorem $\ref{thm7},$ it follows that 
\begin{equation}\label{eq14}
\eta_{\Omega}^{h(Y)}(w)=\eta_{\Omega}(w)
\end{equation} 
for all $w$ in $\Omega.$ Since $Y$ and $f(Y)$ are conformally homeomorphic by the map $f,$ by Corollary~\ref{cor4.17}, we have
\begin{equation}\label{eq15}
\eta_{\Omega}^{Y}(w)=\eta_{\Omega}^{h(Y)}(w)
\end{equation}
for all $w$ in $\Omega.$ Combining $\eqref{eq14}$ and $\eqref{eq15},$ we obtain
\[ \eta_{\Omega}^{Y}(w)=\eta_{\Omega}(w), \]
as desired.
\end{proof}

\bigskip
\noindent 
{\bf Acknowledgement.}
The authors would like to thank the referees for their careful reading of the earlier versions of the manuscript and useful remarks. 
The research work of Arstu is 
supported by CSIR-UGC $($Grant No: 21/06/2015(i)EU-V$)$ and
of S. K. Sahoo is 
supported by NBHM, DAE $($Grant No: $2/48 (12)/2016/${\rm NBHM (R.P.)/R \& D II}/$13613)$.
The authors express their gratitude to Professor Toshiyuki Sugawa for useful discussion on this topic.

\end{document}